\documentclass[12pt]{article}
\usepackage{graphics,amsmath,amssymb}
\usepackage{amsthm}
\usepackage{amsfonts}
\usepackage{latexsym}
\usepackage{amscd}

\newtheorem{theorem}{Theorem}[section]
\newtheorem{lemma}{Lemma}[section]
\newtheorem{corollary}{Corollary}[section]
\newtheorem{conjecture}{Conjecture}[section]

\begin{document}
\begin{center}
\title{A nonconstructive Proof to show the Convergence of the $n^{th}$ root of diagonal Ramsey Number $r(n, n)$} 
\author{\textbf{Robert J. Betts}}
\maketitle
\emph{Department of Computer Science and the Division of Continuing Education, University of Massachusetts Lowell Campus, One University Avenue, Lowell, Massachusetts 01854 Robert\_Betts271828@umb.edu\\ www.continuinged.uml.edu}
\end{center}
\begin{abstract}
Does the $n^{th}$ root of the diagonal Ramsey number converge to a finite limit? The answer is yes. We show the sequence of $n^{th}$ roots does converge by showing one can express it as a product of two factors, the first factor being a known convergent sequence and the second factor being an absolutely convergent infinite series for each $n$. One also can express it where one factor is convergent as a sequence and the other being a convergent sequence defined by a uniformly convergent sequence of complex functions holomorphic within the interior of the unit disc on the complex plane. From that we can draw the needed conclusion from the next step: The limit of a product of two convergent sequences is equal to the product of the two respective limits. Our motivation solely is to prove the conjecture as a problem in search of a solution, not to establish some deep theory about graphs. A second question is: If the limit exists what is it? At the time of this writing the understanding is the proofs sought need not be constructive. Here we show by nonconstructive proofs that the $n^{th}$ root of the diagonal Ramsey number converges to a finite limit. We also show that the limit of the $j^{th}$ root of the diagonal Ramsey number is two, where positive integer $j$ depends upon the Ramsey number.\footnote{MSC (2010) Primary: 26A03, Secondary: 26A12.}
\end{abstract}
\section{Introduction}
A \emph{classical Ramsey number}~\cite{Rosen} \(r(m, n) = k\) is the least positive integer $k$, such that any graph $G$ with $k$ vertices either will have a complete subgraph $K_{m}$ or else its complement will have a complete subgraph $K_{n}$. Ramsey numbers indicate the existence of order even within randomness. In experimental physics there has been some success in the development of quantum algorithms that can compute Ramsey numbers such as $r(3, 3)$, through the use of quantum annealing.~\cite{Bian}. \\
\indent At present the diagonal Ramsey numbers $r(n)$ where \(n \geq 5\) and Ramsey numbers $r(m, n)$ for all \(3 \leq m \leq n\) still are unknown~\cite{Chartrand and Lesniak} (See page 356 for a listing of the nine known Ramsey numbers). To date there exists no known recurrence formula or generating function by which one obtains all Ramsey numbers. Yet with a little analysis one can show that these numbers are bounded and we can find even a necessary and sufficient condition for which $r(n)^{1/n}$ will converge to two within the closed, compact subset $[\sqrt{2}, 4]$ on the real line (Section 3).
\subsection{Upper and lower Bounds on $r(n, n)$}
Let $c$ be some positive real constant (See Section 3, Subsection 3.1). To date Thomason~\cite{Chung and Graham}, found the finest upper bound and Spencer~\cite{Chung and Graham}, the lower bound for diagonal Ramsey number~\cite{Rosen}~\cite{Chartrand and Lesniak},~\cite{Graham},~\cite{Graham and Rothschild},~\cite{Graham and Spencer}~\cite{Landman and Robertson}
$$
r(n, n) \equiv r(n),
$$
(when \(m = n\)) through the application of the Lov\'{a}sz local lemma~\cite{Chung and Graham}~\cite{website1},~\cite{website2}, so that
\begin{equation}
\frac{\sqrt{2}}{e}n2^{n/2} < r(n) < n^{- 1/2 + c/\sqrt{log \: n}}{2n - 2 \choose n - 1}.
\end{equation}
There exist probabilistic proofs of the lemma in the literature~\cite{Graham and Rothschild} (See pages 94--96). With the bound in Eqtn. (1) it follows that \(r(n)^{1/n} \in [\sqrt{2}, 4)\). \\
\indent To prove that \(r(n)^{1/n} \in [\sqrt{2}, 4)\) for large $n$ we can rewrite Eqtn. (1) with the substitution
\begin{equation}
{2n - 2 \choose n - 1} = nC_{n - 1},
\end{equation}
where $C_{n - 1}$ is the $n - 1^{st}$ Catalan number. This way we transform Eqtn. (1) into
\begin{equation}
\frac{\sqrt{2}}{e}n2^{n/2} < r(n) < n \cdot n^{- 1/2 + c/\sqrt{log \: n}}C_{n - 1}.
\end{equation}
With the further substitution 
\begin{equation}
C_{n - 1} \sim \frac{4^{n - 1}}{(n - 1)^{3/2}\sqrt{\pi}}
\end{equation}
for large $n$ one also then can derive as \(n \rightarrow \infty\)~\cite{Chung and Graham}, by the substitution of Eqtn. (4) into Eqtn. (3),
\begin{equation}
\sqrt{2} < r(n)^{1/n} < 4.
\end{equation}
There is a way to show nonconstructively that the $n^{th}$ root of $r(n)$ converges to a finite limit, through the use of analysis. 
\section{The Convergence of $r(n)^{1/n}$}
\subsection{Diagonal Ramsey Numbers $r(n)$ as a Sequence of Numbers on the real Line}
For each fixed $n$ the Ramsey number $r(n)$ has a lower bound of $\sqrt{2^{n}}$~\cite{Chartrand and Lesniak} and an upper bound of $4^{n}$. In fact for all $n$,
\begin{equation}
2^{3/2} < r(3) < r(4) < r(5) < \cdots < r(n) < 4^{n}.
\end{equation}
The sequence of Ramsey numbers \(r(3), r(4), r(5), \ldots\) at the very least, is monotone nondecreasing on $\mathbb{R}$. One does not know yet whether or not the sequence 
\begin{equation}
r(3)^{1/3}, r(4)^{1/4}, r(5)^{1/5}, \ldots
\end{equation}
is monotone for an infinite number of terms. Therefore in this paper \emph{we did not assume nor do we need even to assume, that the sequence of $n^{th}$ roots of diagonal Ramsey numbers is monotone increasing}. In fact if one just takes the time to read carefully through Sections 2--5, one will see that nowhere in our actual arguments in this paper do we use terms such as ``monotone," ``monotone strictly increasing," ``monotonicity," etc. What we do in Section 2 (See Theorem 2.2) and in Section 5 is to show some integer $M$ exists such that, for all \(n \geq M\), \(\lim_{n \rightarrow \infty}r(n)^{\frac{1}{n}} = L < \infty\) holds. In Section 2 we show also that the sequence $(r(n)^{\frac{1}{n}})$ can be rewritten as a product, where the first factor in the product is a sequence known to be convergent and where the second factor is an infinite binomial series already known to be absolutely convergent inside the unit disc on $\mathbb{C}$ for each $n$ and as \(n \rightarrow \infty\) (See Corollary 2.4 and Eqtns. (49)--(50). See also Theorem 5.1.). In Section 5 we show (See Lemma 5.1, Lemma 5.2 and Lemma 5.3) that one can write $(r(n)^{\frac{1}{n}})$ as the product of two convergent sequences, where one of the factors in this product is a sequence known already to be convergent, and where the second factor is a uniformly convergent sequence within \(|z| < 1\) and even convergent on the boundary of the unit disc (this follows from a proof by K. Knopp) when we restrict our attention to principal parts, since \(1/n > 0\)~\cite{Knopp}.
\subsection{Behavior of $r(n)^{1/n}$ for large $n$}
\indent Here we offer a proof of the following Conjecture~\cite{Chung and Graham}:
\begin{conjecture}
The limit \(\lim_{n \rightarrow \infty}r(n)^{1/n}\) exists and is finite within $[\sqrt{2}, 4]$.
\end{conjecture}
Monotonicity and boundedness are not the only means by which to determine if $r(n)^{1/n}$ has a finite limit on $[\sqrt{2}, 4]$. Here in this Subsection we show a means by which \(r(n)^{1/n} \rightarrow L \in \mathbb{R}\) holds for some real $L$, for infinitely many $n$.
\begin{theorem}
Let \(\varepsilon_{1, n + 1}, \varepsilon_{2, n}, \varepsilon_{3, n}, \varepsilon_{4, n + 1} \in \mathbb{R}\) be terms for any four sequences of real numbers with the terms depending upon each $n$, such that 
\begin{eqnarray}
\frac{r(n + 1)^{\frac{1}{n + 1}}}{r(n)^{1/n}}&=&\frac{{2n \choose n}^{\frac{1}{n + 1}} + \varepsilon_{1, n + 1}}{{2(n - 1) \choose n - 1}^{\frac{1}{n}} + \varepsilon_{2, n}},\\
\frac{r(n)^{\frac{1}{n}}}{r(n + 1)^{\frac{1}{n + 1}}}&=&\frac{{2(n - 1) \choose n - 1}^{\frac{1}{n}} + \varepsilon_{3, n}}{{2n\choose n}^{\frac{1}{n + 1}} + \varepsilon_{4, n + 1}},
\end{eqnarray}
where 
\begin{eqnarray}
r(n + 1)^{\frac{1}{n + 1}} + r(n)^{\frac{1}{n}}&=&{2n \choose n}^{\frac{1}{n + 1}} + {2(n - 1) \choose n - 1}^{\frac{1}{n}} + \varepsilon_{1, n + 1} + \varepsilon_{2, n},\\
r(n)^{\frac{1}{n}} + r(n + 1)^{\frac{1}{n + 1}}&=&{2(n - 1) \choose n - 1}^{\frac{1}{n}} + {2n \choose n}^{\frac{1}{n + 1}} + \varepsilon_{3, n} + \varepsilon_{4, n + 1}
\end{eqnarray}
Suppose further that each of the limits 
\begin{eqnarray}
& &\lim_{n\rightarrow \infty}\varepsilon_{1, n + 1},\nonumber \\
& &\lim_{n \rightarrow \infty}\varepsilon_{2, n},\nonumber \\ 
& &\lim_{n\rightarrow \infty}\varepsilon_{3, n},\nonumber \\ 
& &\lim_{n \rightarrow \infty}\varepsilon_{4, n + 1},\nonumber
\end{eqnarray}
converges to the same finite number \(\varepsilon \in \mathbb{R}\). Then for all \(n \geq M\) sufficiently large enough where $M$ is a very large positive integer,
\begin{equation}
r(n)^{\frac{1}{n}} \simeq r(n + 1)^{\frac{1}{n + 1}}.
\end{equation}
That is, for all $n$ sufficiently large enough, $r(n + 1)^{\frac{1}{n + 1}}$ and $r(n)^{1/n}$ are asymptotically equal~\cite{Knopp}, where the symbol $\simeq$ here denotes~\cite{Knopp} ``is asymptotically equal to."
\end{theorem}
\begin{proof}
Each diagonal Ramsey number $r(n)$ is bounded above by the $n - 1$--st central binomial coefficient and $r(n + 1)$ is bounded above by the $n$--th central binomial coefficient~\cite{Chung and Graham}. For each $n$ there exists a real number $\varepsilon_{1, n + 1}$, such that 
\begin{equation}
r(n + 1)^{\frac{1}{n + 1}} = {2n \choose n}^{\frac{1}{n + 1}} + \varepsilon_{1, n + 1}.
\end{equation}
As terms $\varepsilon_{1, n + 1}$ of a convergent sequence, this means
\begin{equation}
\lim_{n \rightarrow \infty}r(n + 1)^{\frac{1}{n + 1}} = \lim_{n \rightarrow}{2n \choose n}^{\frac{1}{n + 1}} + \lim_{n \rightarrow \infty}\varepsilon_{1, n + 1}
\end{equation}
is finite. For each $n$ there exists a real number $\varepsilon_{2, n}$, such that 
\begin{equation}
r(n)^{\frac{1}{n}} = {2(n - 1) \choose n - 1}^{\frac{1}{n}} + \varepsilon_{2, n}.
\end{equation}
In addition as terms $\varepsilon_{2, n}$ of a convergent sequence, this means
\begin{equation}
\lim_{n \rightarrow \infty}r(n)^{\frac{1}{n}} = \lim_{n \rightarrow}{2(n - 1) \choose n - 1}^{\frac{1}{n}} + \lim_{n \rightarrow \infty}\varepsilon_{2, n}
\end{equation}
also is finite. By exactly the same reasoning with \(\varepsilon_{3, n}, \varepsilon_{4, n + 1}\),
\begin{equation}
\lim_{n \rightarrow \infty}r(n)^{\frac{1}{n}} = \lim_{n \rightarrow}{2(n - 1) \choose n - 1}^{\frac{1}{n}} + \lim_{n \rightarrow \infty}\varepsilon_{3, n},
\end{equation}
and
\begin{equation}
\lim_{n \rightarrow \infty}r(n + 1)^{\frac{1}{n + 1}} = \lim_{n \rightarrow}{2n \choose n}^{\frac{1}{n + 1}} + \lim_{n \rightarrow \infty}\varepsilon_{4, n + 1},
\end{equation}
also are finite limits. Here we demonstrate the result of this.\\
\indent From Eqtn. (8),
\begin{eqnarray}
& &\lim_{n \rightarrow \infty}\frac{r(n + 1)^{\frac{1}{n + 1}}}{r(n)^{\frac{1}{n}}}\\
&=&\lim_{n\rightarrow \infty}\frac{{2n \choose n}^{\frac{1}{n + 1}} + \varepsilon_{1, n + 1}}{{2(n - 1) \choose n - 1}^{\frac{1}{n}} + \varepsilon_{2, n}}\\
&=&\lim_{n\rightarrow \infty}\frac{\left(\frac{\sqrt{2\pi(2n)}(2n)^{2n}}{2\pi \cdot n \cdot n^{2n}}\right)^{\frac{1}{n + 1}} + \varepsilon_{1, n + 1}}{\left(\frac{\sqrt{2\pi(2(n - 1))}(2(n - 1))^{(2(n - 1))}}{2\pi \cdot(n - 1)\cdot(n - 1)^{2(n - 1)}}\right)^{\frac{1}{n}} + \varepsilon_{2, n}}  \\
&=&\frac{4 + \varepsilon}{4 + \varepsilon}\nonumber \\
&=&1.
\end{eqnarray}
We applied Stirling's approximation for large $(2n)!$, $n!$, $(2(n - 1))!$, $(n - 1)!$, in Eqtns. (21)--(22) and in Eqtns. (26)--(27). This shows that for all sufficiently large enough $n$ the quotient
\begin{equation}
\frac{r(n + 1)^{\frac{1}{n + 1}}}{r(n)^{1/n}},
\end{equation}
tends to one.\\
\indent Similarly for Eqtn. (9),
\begin{eqnarray}
& &\lim_{n \rightarrow \infty}\frac{r(n)^{\frac{1}{n}}}{r(n + 1)^{\frac{1}{n + 1}}}\\
&=&\lim_{n\rightarrow \infty}\frac{{2(n - 1) \choose n - 1}^{\frac{1}{n}} + \varepsilon_{3, n}}{{2n \choose n}^{\frac{1}{n + 1}} + \varepsilon_{4, n + 1}}\\
&=&\frac{4 + \varepsilon}{4 + \varepsilon} \\
&=&1.
\end{eqnarray}
So for sufficiently large enough $n$, 
\begin{equation}
\frac{r(n)^{\frac{1}{n}}}{r(n + 1)^{\frac{1}{n + 1}}} = 1.
\end{equation}
From Eqtns. (19)--(28) it is clear that \(r(n + 1)^{\frac{1}{n + 1}} \simeq r(n)^{1/n}\) must be true somewhere on $[0, 4]$ (The reason this interval is chosen will be made clear later) for all \(n \geq M\) sufficiently large enough, where $M$ is some very large positive integer.
\end{proof}
Another interpretation of Eqtns. (19)--(28) is that, for all $n$ sufficiently large enough, \(r(n)^{\frac{1}{n}} = O(r(n + 1)^{\frac{1}{n + 1}})\), \(r(n + 1)^{\frac{1}{n + 1}} = O(r(n)^{\frac{1}{n}})\), meaning~\cite{Wiki} 
$$
r(n)^{1/n} = \Theta(r(n + 1)^{\frac{1}{n + 1}}).
$$ 
We also have that
$$
\lim_{n \rightarrow \infty} sup \left|\frac{r(n + 1)^{\frac{1}{n + 1}}}{r(n)^{\frac{1}{n}}}\right| = 1 < \infty,
$$
$$
\lim_{n \rightarrow \infty} sup \left|\frac{r(n)^{\frac{1}{n}}}{r(n + 1)^{\frac{1}{n + 1}}}\right| = 1 < \infty.
$$
The next step is to characterize $r(n)^{1/n}$ as being a sequence that does have the property of convergence to some finite value within some compact set on $\mathbb{R}$.
\begin{corollary}
Let \(M_{1} \geq M, M_{2} \geq M\) be true such that Theorem 2.1 holds for very large positive integers $M_{1}$, $M_{2}$, where $M$ also is very large, and for all sufficiently large enough $n$. Then $r(n)^{1/n}$ is a Cauchy sequence on $[0, 4]$ for all such $n$.
\end{corollary}
\begin{proof}
When $n$ is sufficiently large enough such that \(r(n)^{1/n} \simeq r(n + 1)^{\frac{1}{n + 1}}\) the terms of the sequence must get arbitrarily closer and closer to each other, such that for any \(\epsilon > 0\), 
\begin{equation}
d(r(n)^{1/n}, r(n + 1)^{\frac{1}{n + 1}}) < \epsilon, 
\end{equation}
is true. Let \(n \geq M_{1}, m \geq n + 1 \geq M_{2} \geq M_{1} + 1\) and \(\epsilon_{1} > 0\) any real number. Then from Theorem 2.1 and on $[0, 4]$ and with the usual topology on $\mathbb{R}$,
\begin{eqnarray}
|r(n)^{1/n} - r(m)^{1/m}|&<&\epsilon_{1}, \forall \epsilon_{1} > 0.\\
|r(m)^{1/m} - r(n)^{1/n}|&<&\epsilon_{1}, \forall \epsilon_{1} > 0,
\end{eqnarray}
since
$$
|r(n)^{1/n} - r(m)^{1/m}| = \left|{2(n - 1) \choose n - 1}^{1/n} - {2(m - 1) \choose m - 1}^{1/m} + \varepsilon_{1, n} - \varepsilon_{3, m}\right| < \epsilon_{1},
$$
$$
|r(m)^{1/m} - r(n)^{1/n}| = \left|{2(m - 1) \choose m - 1}^{1/m} - {2(n - 1) \choose n - 1}^{1/n} + \varepsilon_{3, m} - \varepsilon_{1, n}\right| < \epsilon_{1}.
$$
Then the sequence $r(n)^{1/n}$ must be Cauchy on $[0, 4]$ by definition for some large integer $M$ such that \(n, m > M\).
\end{proof}
The differences in Eqtns. (30)--(31) tend to zero for all $n$ sufficiently large enough. This is why we considered the larger point set $[0, 4]$. With the following proof we assert nothing about the actual value for the finite limit for $r(n)^{1/n}$ on $[0, 4]$. Nor can anyone claim that the limit is zero. In fact \(r(n)^{1/n} \in [\sqrt{2}, 4] \subset [0, 4]\) where the greatest lower bound is $\sqrt{2}$. We show only that a finite limit \(r(n)^{1/n} \rightarrow L \in [0, 4]\) does exist somewhere on $[0, 4]$. 
\begin{corollary}
For all $n$ sufficiently large enough such that Theorem 2.1 holds, the sequence $r(n)^{1/n}$ is convergent on $[0, 4]$.
\end{corollary}
\begin{proof}
By Theorem 2.1 and by Corollary 2.1 the sequence $r(n)^{1/n}$ is Cauchy on the compact set $[0, 4]$ which also is a complete metric space on $\mathbb{R}$, and every Cauchy sequence within a closed and bounded set converges.
\end{proof}
\begin{theorem}
There exists always, some \(\varepsilon_{n} \in \mathbb{R}\) depending upon $n$ and even for infinitely many $n$ and positive integers \(M_{1}, M\) where \(M \geq M_{1}\), such that for real number \(\epsilon > 0\), for all integer \(n \geq M\) and where $L$ is some positive real number, the limit
\begin{equation}
\lim_{n \rightarrow \infty}r(n)^{\frac{1}{n}} = L < \infty
\end{equation}
exists.
\end{theorem}
\begin{proof}
We can show by the use of two lemmas, Lemma 5.1 and Lemma 5.3 (See Section 5) that $r(n)^{\frac{1}{n}}$ can be expressed as the product of two factors, namely
$$
r(n)^{\frac{1}{n}} = {2(n - 1) \choose n - 1}^{\frac{1}{n}}f_{n}\left(\frac{|\varepsilon_{n}|}{{2(n - 1) \choose n - 1}}\right),
$$
where the first factor is convergent and the second is from a function $f_{n}(z)$ analytic inside the disc \(|z| < 1\) and which is uniformly convergent inside this disc (See Lemma 5.1, Lemma 5.3, Section 5). However at the present time we wish to prove Theorem 2.2 by other means.\\
\indent The diagonal Ramsey number $r(n)$ is bounded above~\cite{Chung and Graham}, as
\begin{equation}
r(n) \leq {2(n - 1) \choose n - 1},
\end{equation}
which means, certainly and for each $n$ as \(n \rightarrow \infty\), 
$$
{2(n - 1) \choose n - 1} - r(n) \geq 0.
$$
So there has got to exist real $\varepsilon_{n}$  less than or equal to zero (infinitely often if need be) and for each $n$ as \(n \rightarrow \infty\), for which
\begin{equation}
r(n) = {2(n - 1) \choose n - 1} + \varepsilon_{n} \leq {2(n - 1) \choose n - 1},
\end{equation}
always is true. Thus
\begin{eqnarray}
r(n)^{\frac{1}{n}}&=&\left({2(n - 1) \choose n - 1} + \varepsilon_{n}\right)^{\frac{1}{n}} \Longrightarrow \\
\lim_{n \rightarrow \infty}r(n)^{\frac{1}{n}}&=&\lim_{n \rightarrow \infty}\left({2(n - 1) \choose n - 1} + \varepsilon_{n}\right)^{\frac{1}{n}}.
\end{eqnarray}
But since \(\varepsilon_{n} \leq 0\) must be true for each $n$ and infinitely often, we have
\begin{equation}
r(n)^{\frac{1}{n}} = \left({2(n - 1) \choose n - 1} + \varepsilon_{n}\right)^{\frac{1}{n}} \leq {2(n - 1) \choose n - 1}^{\frac{1}{n}}
\end{equation}
which implies
\begin{eqnarray}
\lim_{n \rightarrow \infty}r(n)^{\frac{1}{n}}&=&\lim_{n \rightarrow \infty}\left({2(n - 1) \choose n - 1} + \varepsilon_{n}\right)^{\frac{1}{n}}\\
                                             &\leq&\lim_{n \rightarrow \infty}{2(n - 1) \choose n - 1}^{\frac{1}{n}}\nonumber \\
                                             &\approx&4,
\end{eqnarray}
which one obtains by applying Stirling's approximation (for large $n$) to 
$$
(2(n - 1))!, (n - 1)!, 
$$
and to the central binomial coefficient.\\
\indent Now for all \(n \geq M\) where $M$ is some very large integer, we show there exists some \(\epsilon > 0\) such that the limit $L$ exists (i.e., by definition of ``limit") in Eqtn.(32). Since \(\varepsilon_{n} \leq 0\) is true infinitely often, meaning for each of infinitely many integers $n$ in Eqtn. (34), then for all \(n \geq M\),
\begin{eqnarray}
|r(n)^{\frac{1}{n}} - L|&=&\left|\left({2(n - 1) \choose n - 1} + \varepsilon_{n}\right)^{\frac{1}{n}} - L\right| \\
                        &\leq&\left|{2(n - 1) \choose n - 1}^{\frac{1}{n}} - 4  \right| < \epsilon.\nonumber               
\end{eqnarray}
In fact replacing $\varepsilon_{n}$ with $-|\varepsilon_{n}|$, let $M_{1}$ be any large positive integer, such that
$$
M_{1} \gg \left\lceil\left|\frac{\log\left(1 - \frac{|\varepsilon_{n}|}{{2(n - 1) \choose n - 1}}\right)}{\log \frac{\epsilon}{4}}\right|\right\rceil,
$$
where, for all \(n > M_{1}\),
$$
\left|\left(1 - \frac{|\varepsilon_{n}|}{{2(n - 1) \choose n - 1}}\right)^{\frac{1}{n}} - \frac{L}{4}\right| < \frac{\epsilon}{4}.
$$
It follows that
\begin{eqnarray}
& &\left|\left({2(n - 1) \choose n - 1} + \varepsilon_{n}\right)^{\frac{1}{n}} - L\right|\nonumber \\
&=&\left|\left({2(n - 1) \choose n - 1} - |\varepsilon_{n}|\right)^{\frac{1}{n}} - L\right|\nonumber \\
&=&\left|{2(n - 1) \choose n - 1}^{\frac{1}{n}}\left(1 - \frac{|\varepsilon_{n}|}{{2(n - 1) \choose n - 1}}\right)^{\frac{1}{n}} - 4\cdot\frac{L}{4}\right|\nonumber\\
&<&4\cdot\frac{\epsilon + L}{4} - 4\cdot \frac{L}{4} = \epsilon.\nonumber
\end{eqnarray}
Let \(M \geq M_{1}\). Then for any \(n \geq M\), there exists \(\epsilon > 0\) such that if \(L \in (\sqrt{2}, 4)\) as \(n \rightarrow \infty\), the finite limit in Eqtn. (32) and in Eqtns. (38)--(39) will hold, since
\begin{equation}
|r(n)^{\frac{1}{n}} - L| = \left|\left({2(n - 1) \choose n - 1} + \varepsilon_{n}\right)^{\frac{1}{n}} - L\right| < \epsilon.
\end{equation}
So there exists some open neighborhood 
\begin{equation}
(L -\epsilon, L + \epsilon),
\end{equation}
of $L$, for which \(r(n)^{\frac{1}{n}} \in (L -\epsilon, L + \epsilon)\) is true infinitely often as \(n \rightarrow \infty\). This indicates as \(n \rightarrow \infty\) the limit \(r(n)^{1/n} \rightarrow L < 4\) must be to some finite real number \(L \in [\sqrt{2}, 4]\), since g.l.b. \(r(n)^{\frac{1}{n}} = \sqrt{2}\) and l.u.b. \(r(n)^{\frac{1}{n}} = 4\).
\end{proof}
If there exists no \(\varepsilon_{n} \leq 0\) infinitely often in Eqtns. (37)--(39) such that the limit in Eqtn. (38) exists and is finite and such that
\begin{equation}
\left|\left({2(n - 1) \choose n - 1} + \varepsilon_{n}\right)^{\frac{1}{n}} - L\right| < \epsilon
\end{equation}
holds for any \(\epsilon > 0\), then Eqtn. (32) is false. This means Theorem 2.2 is not a ``trivial" result.\\
\indent As an alternative to the usage
$$
r(n) = {2(n - 1) \choose n - 1} + \varepsilon_{n},
$$
in Eqtn. (34) we instead can write
$$
r(n) = {2(n - 1) \choose n - 1} - |\varepsilon_{n}|,
$$
where it is understood that \(\varepsilon_{n} \leq 0\) and where
$$
|\varepsilon_{n}| = {2(n - 1) \choose n - 1} - r(n).
$$
Henceforth we shall adopt this usage, namely the use of $|\varepsilon_{n}|$ instead of $\varepsilon_{n}$, where \(\varepsilon_{n} \leq 0\). 
\begin{corollary}
Let \(n \geq M\) where $M$ is a very large integer, and let \(\varepsilon_{n} \leq 0\) be as described in Theorem 2.2 and in Eqtns. (34)--(37). Then 
\begin{equation}
r(n)^{\frac{1}{n}} \approx {2(n - 1) \choose n - 1}^{\frac{1}{n}}\left(1 - {\frac{1}{n} \choose 1}\left(\frac{|\varepsilon_{n}|}{{2(n - 1) \choose n - 1}}\right) + O\left(\left(\frac{|\varepsilon_{n}|}{{2(n - 1) \choose n - 1}}\right)^{2}\right)\right)
\end{equation}
\begin{equation}
\leq {2(n - 1) \choose n - 1}^{\frac{1}{n}}.
\end{equation}
\end{corollary}
\begin{proof}
For all large \(n \geq M\),
\begin{eqnarray}
r(n)^{\frac{1}{n}}&=&\left({2(n - 1) \choose n - 1} - |\varepsilon_{n}|\right)^{\frac{1}{n}} \\
                  &=&{2(n - 1) \choose n - 1}^{\frac{1}{n}}\left(1 - \frac{|\varepsilon_{n}|}{{2(n - 1) \choose n - 1}}\right)^{\frac{1}{n}} \\
                  &=&{2(n - 1) \choose n - 1}^{\frac{1}{n}}\left(1 + \sum_{i = 1}^{\infty}(-1)^{i}{\frac{1}{n} \choose i}\left(\frac{|\varepsilon_{n}|}{{2(n - 1) \choose n - 1}}\right)^{i}\right)\nonumber \\
                  &\approx&{2(n - 1) \choose n - 1}^{\frac{1}{n}}\left(1 - {\frac{1}{n} \choose 1}\left(\frac{|\varepsilon_{n}|}{{2(n - 1) \choose n - 1}}\right) + O\left(\left(\frac{|\varepsilon_{n}|}{{2(n - 1) \choose n - 1}}\right)^{2}\right)\right)\nonumber \\
                  &\leq&{2(n - 1) \choose n - 1}^{\frac{1}{n}}.      
\end{eqnarray}
\end{proof}
At the present time one does not know the value of $r(5)$. One does not know even any algorithm by which to determine by general means $r(M)$ where $M$ is a very large integer. Nevertheless we still can use Corollary 2.3 to determine $r(n)^{\frac{1}{n}}$ to first order approximation when one knows $r(n)$. Let \(n = 4\). Then by Corollary 2.3,
\begin{eqnarray}
r(4) = 18&\Longrightarrow&|\varepsilon_{4}| = 2\nonumber \\
         &\Longrightarrow&r(4)^{\frac{1}{4}} \approx {6 \choose 3}^{\frac{1}{4}}\left(1 - {\frac{1}{4}}\cdot \frac{2}{{6 \choose 3}}\right)\nonumber \\
         &=&20^{\frac{1}{4}}\cdot \frac{39}{40} = 2.0618\cdots\nonumber
\end{eqnarray}
Next we establish that the limit \(r(n)^{\frac{1}{n}} \rightarrow L\) in Eqtns. (46)--(48) must converge as \(n \rightarrow \infty\), by using the result in the proof to Corollary 2.3.
\begin{corollary}
For each $n$ and as \(n \rightarrow \infty\), the limit
\begin{equation}
\lim_{n \rightarrow \infty}r(n)^{\frac{1}{n}} = \lim_{n\rightarrow \infty}{2(n - 1) \choose n - 1}^{\frac{1}{n}}\left(1 + \sum_{i = 1}^{\infty}(-1)^{i}{\frac{1}{n} \choose i}\left(\frac{|\varepsilon_{n}|}{{2(n - 1) \choose n - 1}}\right)^{i}\right),
\end{equation}
exists and is finite.
\end{corollary}
\begin{proof}
We derived the product with the binomial series expansion in Eqtn. (49) already in Eqtns. (46)--(48) in the proof to Corollary 2.3. The expression
$$
\left(1 - \frac{|\varepsilon_{n}|}{{2(n - 1) \choose n - 1}}\right)^{\frac{1}{n}} = 1 + \sum_{i = 1}^{\infty}(-1)^{i}{\frac{1}{n} \choose i}\left(\frac{|\varepsilon_{n}|}{{2(n - 1) \choose n - 1}}\right)^{i},
$$
is equal to the power series expansion
\begin{equation}
(1 - z)^{\frac{1}{n}} = 1 + \sum_{i = 1}^{\infty}(-1)^{i}{\frac{1}{n} \choose i}z^{i},
\end{equation}
on $\mathbb{C}$ which is both holomorphic (or in analogous terminology, analytic) and absolutely convergent everywhere in the interior of the unit disk \(|z| < 1\), when \(|z| = ||\varepsilon_{n}|/{2(n - 1) \choose n - 1}| < 1\). The values \(z = |\varepsilon_{n}|/{(2(n - 1) \choose n - 1}\) are all on the real line and the function has a branch point at \(z = 1\). Therefore if need be to avoid any multiple values around any branch cut or branch point when we consider a complex function like
$$
w(z) = (1 - z)^{1/n},
$$
we can restrict our attention to values of the function when \(|z| < 1\) on the first Riemann sheet and for the principal branch for each $n$, on $\mathbb{C}$. Or as an alternative we can choose to expand
$$
(1 - x)^{\frac{1}{n}} = 1 + \sum_{i = 1}^{\infty}(-1)^{i}{\frac{1}{n} \choose i}x^{i},
$$
on $\mathbb{R}^{2}$, instead of using the expansion in Eqtn. (50) on $\mathbb{C}$.\\
\indent In the inequality in Eqtn. (34) the expression
\begin{equation}
r(n) = {2(n - 1) \choose n - 1} + \varepsilon_{n},
\end{equation}
on the right hand side cannot be equal to or less than zero because \(r(n) > 0\) for all $n$. Furthermore in Eqtn. (35) and in Eqtns. (46)--(47), $r(n)^{\frac{1}{n}}$ is bounded below by $\sqrt{2}$ for infinitely many $n$ as \(n \rightarrow \infty\). So with \(\varepsilon_{n} \leq 0\) true for each $n$ as \(n \rightarrow \infty\),
\begin{eqnarray}
r(n) = {2(n - 1) \choose n - 1} + \varepsilon_{n} > 0&\Longrightarrow&{2(n - 1) \choose n - 1} - |\varepsilon_{n}| > 0 \\
                                                     &\Longrightarrow&{2(n - 1) \choose n - 1} = r(n) + |\varepsilon_{n}| > 0\nonumber\\
                                                     &\Longrightarrow&0 \leq \frac{|\varepsilon_{n}|}{{2(n - 1) \choose n - 1}} < 1,                          
\end{eqnarray}
which exactly is what one requires for the binomial series to converge to a finite limit, both for each $n$ and as \(n \rightarrow \infty\), on the right hand side of Eqtn. (49). Moreover since \(r(n) > 0\) is true for each and every \(n \in [3, \infty)\), the quotient 
$$
\frac{|\varepsilon_{n}|}{{2(n - 1) \choose n - 1}},
$$
never is equal to one, since
$$
\frac{|\varepsilon_{n}|}{{2(n - 1) \choose n - 1}} = \frac{{2(n - 1) \choose n - 1} - r(n)}{{2(n - 1) \choose n - 1}} < 1.
$$
In Eqtn. (49) the infinite series converges absolutely to a finite limit for each $n$ as \(i \rightarrow \infty\) and for \(0 \leq \frac{|\varepsilon_{n}|}{{2(n - 1) \choose n - 1}} < 1\). What is more as $n$ increases without bound, the infinite series in the second factor in the product on the right hand side in Eqtn. (49) converges absolutely as \(i \rightarrow \infty\) for each $n$ regardless of the value of $n$, for \(0 \leq \frac{|\varepsilon_{n}|}{{2(n - 1) \choose n - 1}} < 1\) and converges~\cite{Knopp} always even as \(n \rightarrow \infty\) (See Chapter 6, Section 25, pages 206--209, and Chapter 12, Theorem 245, Theorem 247 and Theorem 246, where it is made clear that the binomial series converges absolutely on $\mathbb{C}$ for each $1/n$ for the principal values).\\
\indent We demonstrate this more explicitly for those who still choose to doubt what we claim in the previous paragraph with regard to the convergence of the infinite series that appears in Eqtn. (49). \\
\indent From Eqtns. (49)--(50), 
$$
\left|\left(1 - \frac{|\varepsilon_{n}|}{{2(n - 1) \choose n - 1}}\right)^{\frac{1}{n}}\right| = \left|1 + \sum_{i = 1}^{\infty}(-1)^{i}{\frac{1}{n} \choose i}\left(\frac{|\varepsilon_{n}|}{{2(n - 1) \choose n - 1}}\right)^{i}\right|.
$$
Each side is $O(1)$ on $\mathbb{C}$ if we let \(z = \frac{|\varepsilon_{n}|}{{2(n - 1) \choose n - 1}} \in |z| < 1\) for all \(n \gg 3\). In fact each side of the equation is bounded inside that square on $\mathbb{C}$ with its four vertices at the points \(1 + i, 1 - i, -1 + i, -1 - i\)~\cite{Knopp} (Theorem 225, Chapter 12, page 394). By the Bolzano--Weierstrass property, every bounded infinite sequence--such as the one we are dealing with that has the terms 
$$
\left(1- \frac{|\varepsilon_{n}|}{{2(n - 1) \choose n - 1}}\right)^{\frac{1}{n}},
$$
--has no less than one limit point and it just so happens that there is exactly one limit point (See Theorem 5.1, Section 5). In addition it is known that an infinite series such as the one that appears in the product on the right hand side of Eqtn. (49), has a sequence of partial sums that does converge uniformly to a finite limit precisely because for each $n$ and as $n$ increases without bound, each quotient $\frac{|\varepsilon_{n}|}{{2(n - 1) \choose n - 1}}$ remains ever bounded below by zero and above by one. On the other hand if 
$$
\frac{|\varepsilon_{n}|}{{2(n - 1) \choose n - 1}} > 1,
$$
was the case infinitely often in the infinite series on the right hand side of Eqtn. (49), the series would not converge.\\
\indent So for each $n$ and as $n$ goes to infinity in the second factor on the right hand side of Eqtn. (49) the term with the infinite series expansion converges. On the right hand side in Eqtn. (49) we have also in the first product factor, for large $n$,
$$
{2(n - 1) \choose n - 1}^{\frac{1}{n}} \approx 4.
$$
So the limit
$$
\lim_{n \rightarrow \infty}{2(n - 1) \choose n - 1}^{\frac{1}{n}},
$$
on the right of Eqtn. (49) also is finite, so that both products on the right hand side of Eqtn. (49) have finite limits. Therefore in Eqtn. (49) and on the left hand side,
\begin{equation}
\lim_{n\rightarrow \infty}r(n)^{\frac{1}{n}},
\end{equation}
is finite, because it is equal to the limit of a product of two convergent sequences, one
$$
\lim_{n \rightarrow \infty}{2(n - 1) \choose n - 1}^{\frac{1}{n}},
$$
which is a convergent finite limit, and the other 
$$
\lim_{n \rightarrow \infty}\left(1 - \frac{|\varepsilon_{n}|}{{2(n - 1) \choose n - 1}}\right)^{\frac{1}{n}} = \lim_{n \rightarrow \infty}\left(1 + \sum_{i = 1}^{\infty}(-1)^{i}{\frac{1}{n} \choose i}\left(\frac{|\varepsilon_{n}|}{{2(n - 1) \choose n - 1}}\right)^{i}\right),
$$
which also converges on both sides for each $n$ when both \(i \rightarrow \infty\) and \(0 \leq \frac{|\varepsilon_{n}|}{{2(n - 1) \choose n - 1}} < 1\) hold and even as $n$ increases without bound. Therefore both sides of Eqtn. (49) converge to finite limits.
\end{proof}
\begin{corollary}
$$
\lim_{n \rightarrow \infty}\left(1 + \sum_{i = 1}^{\infty}(-1)^{i}{\frac{1}{n} \choose i}\left(\frac{|\varepsilon_{n}|}{{2(n - 1) \choose n - 1}}\right)^{i}\right) = \frac{L}{4}.
$$
\end{corollary}
\begin{proof}
From Corollary 2.4,
\begin{eqnarray}
\lim_{n \rightarrow \infty}r(n)^{\frac{1}{n}}&=&\lim_{n\rightarrow \infty}{2(n - 1) \choose n - 1}^{\frac{1}{n}}\left(1 + \sum_{i = 1}^{\infty}(-1)^{i}{\frac{1}{n} \choose i}\left(\frac{|\varepsilon_{n}|}{{2(n - 1) \choose n - 1}}\right)^{i}\right),\nonumber \\
                                             &\Longrightarrow&L = 4\lim_{n\rightarrow \infty}\left(1 + \sum_{i = 1}^{\infty}(-1)^{i}{\frac{1}{n} \choose i}\left(\frac{|\varepsilon_{n}|}{{2(n - 1) \choose n - 1}}\right)^{i}\right),\nonumber \\
                                             &\Longrightarrow&\lim_{n\rightarrow \infty}\left(1 + \sum_{i = 1}^{\infty}(-1)^{i}{\frac{1}{n} \choose i}\left(\frac{|\varepsilon_{n}|}{{2(n - 1) \choose n - 1}}\right)^{i}\right) = \frac{L}{4}.\nonumber
\end{eqnarray}
\end{proof}
We have found for each $n$, values possible for \(\varepsilon_{1, n + 1}, \varepsilon_{2, n}, \varepsilon_{3, n}, \varepsilon_{4, n + 1}\) in the proof to Theorem 2.1, namely (See Eqtns. (47)--(49))
\begin{eqnarray}
\varepsilon_{1, n + 1}&=&\varepsilon_{4, n + 1}\\
                      &=&{2n \choose n}^{\frac{1}{n + 1}}\sum_{i = 1}^{\infty}(-1)^{i}{\frac{1}{n + 1} \choose i}\left(\frac{|\varepsilon_{n + 1}|}{{2n \choose n}}\right)^{i},\nonumber \\
\varepsilon_{2, n}    &=&\varepsilon_{3, n}\nonumber \\
                      &=&{2(n - 1) \choose n - 1}^{\frac{1}{n}}\sum_{i = 1}^{\infty}(-1)^{i}{\frac{1}{n} \choose i}\left(\frac{|\varepsilon_{n}|}{{2(n - 1) \choose n - 1}}\right)^{i},
\end{eqnarray}
where
\begin{eqnarray}
|\varepsilon_{n + 1}|&=&{2n \choose n} - r(n + 1),\nonumber \\
|\varepsilon_{n}|    &=&{2(n - 1) \choose n - 1} - r(n).
\end{eqnarray}
\indent The radius of convergence for the series in Eqtn. (49) exists when \(|\varepsilon_{n}| < {2(n - 1) \choose n - 1}\), and as we can see from Eqtns. (51)--(53) this is the case. 
\section{\(|\varepsilon_{n}| \not = 0\) is true infinitely often}
Here we consider whether or not
\begin{equation}
|\varepsilon_{n}| = 0,
\end{equation}
is true infinitely often, meaning for each and every $n$ sufficiently large enough. 
\subsection{A finer upper Bound exists on $r(n)$ than ${2(n - 1) \choose n - 1}$, for which \(r(n)^{1/n} \rightarrow L < \infty\)}
It just so happens that a finer upper bound exists on $r(n)$ for some real constant $c$. It is~\cite{Chung and Graham} (See Chapter 2, Equation 2.6, page 9)
\begin{equation}
r(n) < b(n){2(n - 1) \choose n - 1} < {2(n - 1) \choose n - 1},
\end{equation}
where
\begin{equation}
b(n) = n^{-\frac{1}{2} + \frac{c}{\sqrt{\log n}}}.
\end{equation}
This being the case we have that there exists always for each $n$, some real number \(\delta_{n} > 0\), such that
\begin{equation}
b(n){2(n - 1) \choose n - 1} - r(n) = \delta_{n} > 0,
\end{equation}
$$
b(n){2(n - 1) \choose n - 1} - \delta_{n} = r(n) > 0,
$$
\begin{equation}
0 < \frac{\delta_{n}}{b(n){2(n - 1) \choose n - 1}} < 1.
\end{equation}
We certainly will have
$$
r(n)^{\frac{1}{n}} < {2(n - 1) \choose n - 1}^{\frac{1}{n}},
$$
for all $n$ large enough such that \(\frac{c}{\sqrt{\log n}} < 1/2\). One then can show that, in a manner similar to the proof of Theorem 2.2 by substituting $\delta_{n}$ for $|\varepsilon_{n}|$ and for any \(\epsilon > |4 - L|\),
\begin{eqnarray}
|r(n)^{\frac{1}{n}} - L|&=&\left|\left(b(n){2(n - 1) \choose n - 1} - \delta_{n}\right)^{\frac{1}{n}} - L\right|\\
                        &<&\left|{2(n - 1) \choose n - 1}^{\frac{1}{n}} - L\right|\nonumber \\
                        &=&|4 - L| < \epsilon.
\end{eqnarray}
Also the limit
$$
\lim_{n \rightarrow \infty}b(n)^{\frac{1}{n}}{2(n - 1) \choose n - 1}^{\frac{1}{n}}
$$
exists and is finite. So in a manner similar to what we did in the proof to Corollary 2.4,
\begin{eqnarray}
\lim_{n \rightarrow \infty}r(n)^{\frac{1}{n}}&=&\lim_{n \rightarrow \infty}\left(b(n){2(n - 1) \choose n - 1} - \delta_{n}\right)^{\frac{1}{n}}\\
                                             &=&\lim_{n \rightarrow \infty}\left(b(n){2(n - 1) \choose n - 1}\right)^{\frac{1}{n}}\left(1 + \sum_{i = 1}^{\infty}(-1)^{i}\left(\frac{\delta_{n}}{b(n){2(n - 1) \choose n - 1}}\right)^{i}\right)\nonumber \\
                                             &=&L < \infty,
\end{eqnarray}
$$
0 < \frac{\delta_{n}}{{2(n - 1) \choose n - 1}} < 1.
$$
So
\begin{equation}
r(n) < {2(n - 1) \choose n - 1}
\end{equation}
must be true for each and every $n$ sufficiently large enough and as \(n \rightarrow \infty\), since
\begin{equation}
r(n) = {2(n - 1) \choose n - 1} - |\varepsilon_{n}| < b(n){2(n - 1) \choose n - 1} < {2(n - 1) \choose n - 1},
\end{equation}
establishes a finer upper bound on each $r(n)$ by $b(n){2(n - 1) \choose n - 1}$. This also shows that \(|\varepsilon_{n}| = 0\) cannot be true infinitely often because~\cite{Chung and Graham}
\begin{eqnarray}
r(n)&=&{2(n - 1) \choose n - 1} - |\varepsilon_{n}| < b(n){2(n - 1) \choose n - 1}\\
    &<&{2(n - 1) \choose n - 1}\nonumber \\
    &\Longrightarrow&|\varepsilon_{n}| > 0.
\end{eqnarray}
After the discussion in this Section we are in a position to demonstrate why \(r(n)^{\frac{1}{n}} < 4\) is true whenever $n$ is very large, on the compact set $[\sqrt{2}, 4]$.
\begin{theorem}
Let \(n = M\) be any very large integer so large, that \(\sqrt{\log M} \gg c\) \(\Longrightarrow b(M) \sim M^{-1/2}\) is true in Eqtn. (60). Then for each and every such integer $M$, \(r(M)^{\frac{1}{n}} < 4\).
\end{theorem}
\begin{proof}
From the proof to Corollary 2.3 and to first order approximation,
\begin{equation}
r(M)^{\frac{1}{M}} \approx {2(M - 1) \choose M - 1}^{\frac{1}{M}}\left(1 - {\frac{1}{M} \choose 1}\left(\frac{|\varepsilon_{M}|}{{2(M - 1) \choose M - 1}}\right)\right).
\end{equation}
Then from Eqtn. (71) up to first order approximation and from Eqtns. (59)--(60),
\begin{eqnarray}
r(M)^{\frac{1}{M}}&\approx&{2(M - 1) \choose M - 1}^{\frac{1}{M}}\left(1 - {\frac{1}{M} \choose 1}\left(\frac{|\varepsilon_{M}|}{{2(M - 1) \choose M - 1}}\right)\right) < 4,\\
r(M)^{\frac{1}{M}}&\leq&\left(b(M){2(M - 1) \choose M - 1}\right)^{\frac{1}{M}} < {2(M - 1) \choose M - 1}^{\frac{1}{M}} \approx 4.
\end{eqnarray}
\end{proof}
\begin{theorem}
\begin{equation}
\lim_{n \rightarrow \infty}\left(1 + \frac{|\varepsilon_{n}|}{r(n)}\right)^{\frac{1}{n}} = \frac{4}{L}.
\end{equation}
\end{theorem}
\begin{proof}
From Section 2, \({2(n - 1) \choose n - 1} = \)
\begin{eqnarray}
&               &r(n) + |\varepsilon_{n}|\\
&\Longrightarrow&{2(n - 1) \choose n - 1}^{\frac{1}{n}} = (r(n) + |\varepsilon_{n}|)^{\frac{1}{n}}\nonumber \\
&=              &r(n)^{\frac{1}{n}}\left(1 + \frac{|\varepsilon_{n}|}{r(n)}\right)^{\frac{1}{n}} \\
&\Longrightarrow&\lim_{n \rightarrow \infty}{2(n - 1) \choose n - 1}^{\frac{1}{n}}\\
&=              &\lim_{n \rightarrow \infty}r(n)^{\frac{1}{n}}\left(1 + \frac{|\varepsilon_{n}|}{r(n)}\right)^{\frac{1}{n}}\nonumber \\
&\Longrightarrow&4 = L\lim_{n \rightarrow \infty}\left(1 + \frac{|\varepsilon_{n}|}{r(n)}\right)^{\frac{1}{n}} \\
&\Longrightarrow&\lim_{n \rightarrow \infty}\left(1 + \frac{|\varepsilon_{n}|}{r(n)}\right)^{\frac{1}{n}} = \frac{4}{L}.
\end{eqnarray}
\end{proof}
We can verify Theorem 3.2 by other means, as follows:
\begin{eqnarray}
\frac{|\varepsilon_{n}|}{r(n)}&=&\frac{{2(n - 1) \choose n - 1} - r(n)}{{2(n - 1) \choose n - 1} - |\varepsilon_{n}|}\nonumber\\
                              &=&\frac{{2(n - 1) \choose n - 1}}{{2(n - 1) \choose n - 1} - |\varepsilon_{n}|} - \frac{r(n)}{{2(n - 1) \choose n - 1} - |\varepsilon_{n}|}\nonumber\\
                              &=&\frac{{2(n - 1) \choose n - 1}}{r(n)} - 1.\nonumber
\end{eqnarray}
We just have proved that
$$
\frac{|\varepsilon_{n}|}{r(n)} = \frac{{2(n - 1) \choose n - 1}}{r(n)} - 1.
$$
So by substitution into Eqtn. (74) the same result in Theorem 3.2 follows.
\section{The Meaning of $|\varepsilon_{n}|$}
From Eqtn. (34),
\begin{equation}
{2(n - 1) \choose n - 1} - r(n) = |\varepsilon_{n}| \geq 0.
\end{equation}
Let \(G = K_{r(n)}\) be a graph for which $r(n)$ is the minimum integer such that $G$ has either a clique of size $n$ or an independent set of size $n$. Then $|\varepsilon_{n}|$ is the number of ways to choose $n - 1$ vertices from the $2(n - 1)$ vertices in the vertex set of $K_{2(n - 1)}$, minus the minimum integer for which the graph $G$ will have either a clique of size $n$ for a complete bipartite graph $K_{n}$, or else an independent set of size $n$ for the complement $\overline{K}_{n}$. For each $n$ the integer
\begin{equation}
{2(n - 1) \choose n - 1},
\end{equation}
also is a number that is related to the bipartite dimension of a graph. When \(n = 3, r(3) = 6\) we have \(|\varepsilon_{3}| = 0\), when \(n = 4, r(4) = 18\) we get \(|\varepsilon_{4}| = 2\) and when \(n = n_{0} \geq 5, a \leq r(n_{0}) \leq b\) for some integer g.l.b. $a$ and for some l.u.b. $b$,
\begin{equation}
{2(n_{0} - 1) \choose n_{0} - 1} - b \leq |\varepsilon_{n_{0}}| \leq {2(n_{0} - 1) \choose n_{0} - 1} - a.
\end{equation}
In the limits in Eqtn. (49) one can replace the appearance of $|\varepsilon_{n}|$ on the right hand side if one so wishes, with
\begin{equation}
{2(n - 1) \choose n - 1} - r(n),
\end{equation}
so that
\begin{eqnarray}
\frac{|\varepsilon_{n}|}{{2(n - 1) \choose n - 1}}&=&\frac{{2(n - 1) \choose n - 1} - r(n)}{{2(n - 1) \choose n - 1}} = 1 - \frac{r(n)}{{2(n - 1) \choose n - 1}},\\
\frac{r(n)}{{2(n - 1) \choose n - 1}}&=&\frac{{2(n - 1) \choose n - 1} - |\varepsilon_{n}|}{{2(n - 1) \choose n - 1}} = 1 - \frac{|\varepsilon_{n}|}{{2(n - 1) \choose n - 1}}.\nonumber
\end{eqnarray}
Doing so makes $r(n)$ in Eqtn. (49) appear on both sides of the limit. But this does not suggest either some rule of assignment at work, a function, some kind of mapping nor even recursion. What it means is the following: For every integer \(n \geq 3\) and for every diagonal Ramsey number $r(n)$ there exists some real number $\varepsilon_{n}$, such that
\begin{equation}
|\varepsilon_{n}| = {2(n - 1) \choose n - 1} - r(n).
\end{equation}
So to replace $|\varepsilon_{n}|$ in Eqtn. (49) with
\begin{equation}
{2(n - 1) \choose n - 1} - r(n),
\end{equation}
does not indicate some unknown function nor a mapping nor even some kind of recursion at work. It simply is the substitution of one real number integer $|\varepsilon_{n}|$ with another real number integer ${2(n - 1) \choose n - 1} - r(n)$, by the rules of substitution and ordinary arithmetic among real number field elements \(|\varepsilon_{n}|, r(n), {2(n - 1) \choose n - 1}\).
One also can give a computational description of $|\varepsilon_{n}|$. For each $n$ we can define if one wishes, Eqtn. (85) as being the absolute error (we chose to leave out the absolute value bars on the right hand side, since the value on the right hand side is nonnegative) when we try to approximate ${2(n - 1) \choose n - 1}$ with $r(n)$, or when we try to compare the rate of growth of ${2(n - 1) \choose n - 1}$ with the rate of growth of $r(n)$. Similarly for each $n$ we can define if one wishes,
\begin{equation}
\frac{{2(n - 1) \choose n - 1} - r(n)}{{2(n - 1) \choose n - 1}},
\end{equation}
in Eqtn. (84) as being the relative error (again with the absolute bars removed, for sake of convenience) in the approximation of ${2(n - 1) \choose n - 1}$ by $r(n)$.
\section{The $n^{th}$ root of $r(n)$ expressed as a complex valued function on $\mathbb{C}$}
First we prove a result that verifies the result we obtained previously in Corollary 2.5 (See also Eqtns. (49)--(50)).
\begin{theorem}
For \(n \gg 3\) and as \(n \rightarrow \infty\),
\begin{equation}
\lim_{n \rightarrow \infty}\left(\frac{r(n)}{{2(n - 1) \choose n - 1}}\right)^{\frac{1}{n}} = \frac{L}{4}.
\end{equation}
\end{theorem}
\begin{proof}
Recall that from Eqtn. (84),
\begin{equation}
\left(\frac{r(n)}{{2(n - 1) \choose n - 1}}\right)^{\frac{1}{n}} = \left(1 - \frac{|\varepsilon_{n}|}{{2(n - 1) \choose n - 1}}\right)^{\frac{1}{n}}.
\end{equation}
Since
\begin{equation}
0 < \left(\frac{r(n)}{{2(n - 1) \choose n - 1}}\right)^{\frac{1}{n}} < 1
\end{equation}
is true for \(n \gg 3\), we have, and for some \(L \in [\sqrt{2}, 4]\), 
\begin{equation}
\left|\left(\frac{r(n)}{{2(n - 1) \choose n - 1}}\right)^{\frac{1}{n}} - \frac{L}{4}\right| < 1.
\end{equation}
But then
\begin{equation}
\left(\frac{r(n)}{{2(n - 1) \choose n - 1}}\right)^{\frac{1}{n}} < 1 + \frac{L}{4}.
\end{equation}
This means, since \(0 < \left(\frac{r(n)}{{2(n - 1) \choose n - 1}}\right)^{\frac{1}{n}} < 1\),
\begin{equation}
\frac{1}{n\log\left(1 + \frac{L}{4}\right)} \cdot \log\left(\frac{r(n)}{{2(n - 1) \choose n - 1}}\right) < 1.
\end{equation}
Then choose any real number \(\epsilon \geq 1 + \frac{L}{4}\), and let
\begin{equation}
M > \left\lceil \frac{1}{\log \epsilon} \cdot \log\left(\left|\frac{r(n)}{{2(n - 1) \choose n - 1}}\right|\right)\right\rceil,
\end{equation}
be some positive integer. Then for any real number \(\epsilon \geq 1 + \frac{L}{4}\) we have that for all integer \(n \geq M\) Eqtn. (88) holds for some \(L \in [\sqrt{2}, 4]\), where $n$ depends upon $\epsilon$ but $M$ does not depend upon $\frac{r(n)}{{2(n - 1) \choose n - 1}}$ due to Eqtn. (90) and since in Eqtn. (94), \(0 < \frac{r(n)}{{2(n - 1) \choose n - 1}} < 1\). 
\end{proof}
We know already that
\begin{equation}
\lim_{n \rightarrow \infty}{2(n - 1) \choose n - 1}^{\frac{1}{n}}
\end{equation}
converges to a finite limit. However
\begin{equation}
\lim_{n \rightarrow \infty}r(n)^{\frac{1}{n}} = \lim_{n \rightarrow \infty}{2(n - 1) \choose n - 1}^{\frac{1}{n}}\left(1 - \frac{|\varepsilon_{n}|}{{2(n - 1) \choose n - 1}}\right)^{\frac{1}{n}},
\end{equation}
also converges, first because the limit in Eqtn. (88) is convergent and the second factor in the right hand side of Eqtn. (89) converges uniformly within a simply--connected domain on $\mathbb{C}$. So we show next through the proof of two lemmas, that one can rewrite the limit on the right hand side of Eqtn. (38) so that \(\lim_{n \rightarrow \infty}r(n)^{\frac{1}{n}}\) converges for any \(z = \frac{|\varepsilon_{n}|}{{2(n - 1) \choose n - 1}}\) inside the unit disk on $\mathbb{C}$ because both factors on the right hand side of Eqtn. (89) do, including whenever
$$
z = \frac{|\varepsilon_{n}|}{{2(n - 1) \choose n - 1}},
$$
$$
0 < \left|\frac{|\varepsilon_{n}|}{{2(n - 1) \choose n - 1}}\right| < 1,
$$
is true inside the interior \(|z| < 1\).
\begin{lemma}
\begin{equation}
\lim_{n\rightarrow \infty}\left(1 - \frac{|\varepsilon_{n}|}{{2(n - 1) \choose n - 1}}\right)^{\frac{1}{n}} \not = 0.
\end{equation}
\end{lemma}
\begin{proof}
We remind the reader again of Eqtn. (84). For each \(n \geq 3\) and as \(n \rightarrow \infty\), 
\begin{equation}
r(n)^{\frac{1}{n}} \in (\sqrt{2}, 4),
\end{equation}
and the set of limit points for the sequence with its terms in $(\sqrt{2}, 4)$ is the closure $[\sqrt{2}, 4]$, which is a compact set that contains all its limit points. Therefore since \(0 \not \in [\sqrt{2}, 4]\),
\begin{equation}
\lim_{n\rightarrow \infty}\left(1 - \frac{|\varepsilon_{n}|}{{2(n - 1) \choose n - 1}}\right)^{\frac{1}{n}} \not = 0.
\end{equation}
\end{proof}
\begin{lemma}
For each and every \(n \geq 3\) and as \(n \rightarrow \infty\), restrict for each $n$, all values of the complex valued function
\begin{equation}
f_{n}(z) = \left(1 - z\right)^{\frac{1}{n}}
\end{equation}
inside \(|z| < 1\) to the first Riemann sheet, meaning the one with the principal values on $\mathbb{C}$. Then there exists positive integer $M$ which depends upon positive real $\epsilon$ but not upon $z$, such that for all \(n \geq M\), \(\lim_{n \rightarrow \infty}f_{n}(z)\) converges uniformly inside $\mathbb{C}$ for all \(z \in D \subseteq |z| < 1\), where $D$ is some subset of the unit disc.
\end{lemma}
\begin{proof}
By Lemma 5.1 we are assured that
\begin{equation}
\lim_{n\rightarrow \infty}\left(1 - \frac{|\varepsilon_{n}|}{{2(n - 1) \choose n - 1}}\right)^{\frac{1}{n}} \not = 0,
\end{equation}
does not hold. Let \(g(z) = l\), $l$ real, be some constant function on $D$ for all $z$ such that \(z \in D \subseteq |z| < 1\). We prove uniform convergence on $D$ by showing that \emph{for some} finite positive real number $l$ and for any \(\epsilon > 0\),
\begin{equation}
|f_{n}(z) - g(z)| = |f_{n}(z) - l| < \epsilon,
\end{equation}
is true in such a way that $n$ does not depend upon the choice of \(z \in D \subseteq |z| < 1\).\\
\indent Let \(\epsilon^{\prime} > 0\) be some real number such that, inside $D$,
\begin{equation}
|f_{n}(z) - l| = \left|\left(1 - z\right)^{\frac{1}{n}} - l\right| < \epsilon^{\prime}.
\end{equation}
Then whenever $z$ is restricted to real values inside $D$,
\begin{eqnarray}
& &\left|\left(1 - z\right)^{\frac{1}{n}} - l\right| < \epsilon^{\prime}\\
&\Longrightarrow&\frac{1}{l}\left(1 - z\right)^{\frac{1}{n}} < \frac{\epsilon^{\prime}}{l} + 1,\nonumber \\
&\Longrightarrow&\frac{1}{n}\log\left(1 - z\right) < \log\left(\epsilon^{\prime} + l\right).
\end{eqnarray}
Let
\begin{equation}
C = \log\left(\epsilon^{\prime} + l\right),
\end{equation}
such that
\begin{equation}
\frac{1}{nC}\log\left(1 - z\right) < 1.
\end{equation}
All the possible values for $|z|$ are bounded by one. So let $M$ be any positive integer large enough such that, for any real \(z \in |z| < 1\),
\begin{equation}
M > \left\lceil \frac{1}{C}\log\left(1 - z\right)\right\rceil.
\end{equation}
If $z$ is complex with principal values restricted to the first Riemann sheet so that one avoids branch cuts, then allow
$$
M > \left\lceil \frac{1}{C}\log\left|1 - z\right|\right\rceil.
$$
We see the integer $M$ depends only upon $\epsilon^{\prime}$ (See Eqtn. (98)) but it does not depend upon any \(z \in |z| < 1\). Then for any \(\epsilon \geq \epsilon^{\prime}\) and for all integer \(n \gg M\), we have that 
\begin{equation}
|f_{n}(z) - l| < \epsilon
\end{equation}
holds, which means
\begin{equation}
\lim_{n \rightarrow \infty}f_{n}(z) = l < \infty
\end{equation}
indicates uniform convergence for all \(z \in D \subseteq |z| < 1\).
\end{proof}
\begin{lemma}
Inside \(|z| < 1\) on $\mathbb{C}$,
\begin{equation}
\lim_{n \rightarrow \infty}f_{n}\left(\frac{|\varepsilon_{n}|}{{2(n - 1) \choose n - 1}}\right) = \lim_{n \rightarrow \infty}\left(1 - \frac{|\varepsilon_{n}|}{{2(n - 1) \choose n - 1}}\right)^{\frac{1}{n}},
\end{equation}
converges uniformly whenever \(z = \frac{|\varepsilon_{n}|}{{2(n - 1) \choose n - 1}}\).
\end{lemma}
\begin{proof}
This follows from Lemma 5.2, since inside \(|z| < 1\),
\begin{equation}
0 < z = \frac{|\varepsilon_{n}|}{{2(n - 1) \choose n - 1}} < 1.
\end{equation}
That is, all the points \(z = \frac{|\varepsilon_{n}|}{{2(n - 1) \choose n - 1)}}\) remain inside the region \(|z| < 1\) even as \(n \rightarrow \infty\). 
\end{proof}
From the results in Lemma 5.2 and Lemma 5.3 we can rewrite $r(n)^{\frac{1}{n}}$ as
\begin{equation}
r(n)^{\frac{1}{n}} = {2(n - 1) \choose n - 1}^{\frac{1}{n}}f_{n}\left(\frac{|\varepsilon_{n}|}{{2( n - 1) \choose n - 1}}\right),
\end{equation}
where we have defined $f_{n}(z)$ in Eqtn. (100).
\section{The Limit \(r(n)^{1/j} \rightarrow 2\) is true on $\mathbb{R}$ for some positive integer $j$}
In this section we show that, on $[\sqrt{2}, 4]$, \(r(n)^{1/j} \rightarrow 2\) for \(j \in \mathbb{N}\).
\begin{theorem}
There exists some \(j \in \mathbb{N}\) depending upon $n$, such that for large $n$ and on the compact interval $[\sqrt{2}, 4]$,
\begin{equation}
\lim_{j, n\rightarrow \infty} r(n)^{1/j} = 2.
\end{equation}
\end{theorem}
\begin{proof}
For each $n$ let $t_{n}$ be the largest integer exponent for which $2^{t_{n}}$ divides $r(n)$. Then the Ramsey number $r(n)$ has an expansion into powers of two as
\begin{equation}
r(n) = c_{t_{n}}2^{t_{n}} + c_{t_{n} - 1}2^{t_{n} - 1} + \cdots + c_{0}, \: c_{t_{n}} \not = 0,
\end{equation}
\begin{equation}
c_{t_{n}} = 1, c_{t_{n} - 1}, \cdots, c_{0} \in [0, 1].
\end{equation}
For each $n$ we have the inequality
\begin{equation}
r(n) = 2^{t_{n}} + c_{t_{n} - 1}2^{t_{n} - 1} + \cdots + c_{0} \leq 2^{t_{n} + 1}.
\end{equation}
Therefore 
\begin{equation}
r(n) = 2^{t_{n}} + c_{t_{n} - 1}2^{t_{n} - 1} + \cdots + c_{0} \leq 2^{t_{n} + 1}.
\end{equation}
It follows from Eqtn. (115)--(118), that
\begin{eqnarray}
2^{t_{n}}&\leq&2^{t_{n}} + c_{t_{n} - 1}2^{t_{n} - 1} + \cdots + c_{0} \leq 2^{t_{n} + 1} \Longrightarrow\\
2^{t_{n}}&\leq&r(n) \leq 2^{t_{n} + 1},
\end{eqnarray}
when 
\begin{equation}
r(n) = 2^{t_{n}} + c_{t_{n} - 1}2^{t_{n} - 1} + \cdots + c_{0}.
\end{equation}
For each $t_{n}$, $n$, let \(j \geq max(\{t_{n} + 1\), n\}), where
\begin{equation}
r(n)^{1/j} = (2^{t_{n}} + c_{t_{n} - 1}2^{t_{n} - 1} + \cdots + c_{0})^{1/j}.
\end{equation}
Then with \(t_{n} \leq j - 1, n \leq j\) true always as \(j, n \rightarrow \infty\), taking the $j^{th}$ roots in Eqtn. (121) where $r(n)$ is as given in Eqtn. (122), then taking limits on $[\sqrt{2}, 4]$ as \(j, n \rightarrow \infty\), we have a ``pinching theorem" result,
\begin{eqnarray}
& &\lim_{j, n \rightarrow \infty}2^{\frac{t_{n}}{j}} \leq \lim_{j, n \rightarrow \infty} 2^{\frac{j - 1}{j}} \leq \lim_{j, n \rightarrow \infty} r(n)^{1/j} \leq \lim_{j, n \rightarrow \infty} 2^{\frac{j}{j}}\\
&\Longrightarrow&\lim_{j, n \rightarrow \infty} 2\cdot 2^{-1/j} \leq \lim_{j, n \rightarrow \infty} r(n)^{1/j} \leq \lim_{j, n \rightarrow \infty} 2^{j/j} = 2\nonumber \\
&\Longrightarrow&2 \leq \lim_{j, n \rightarrow \infty} r(n)^{1/j} \leq 2\nonumber \\
&\Longrightarrow&\lim_{j, n \rightarrow \infty} r(n)^{1/j} = 2.
\end{eqnarray}
\end{proof}

\pagebreak

\end{document}